\documentclass[11pt]{article}
\usepackage{amsmath,amsfonts,amssymb,latexsym}
\usepackage{mathrsfs}
\usepackage{verbatim}
\usepackage{latexsym}
\usepackage{amsthm}
\usepackage{amssymb}
\usepackage{amsbsy}
\usepackage{enumerate}
\usepackage{times}

\newtheorem{thm}{Theorem}[section]

\newtheorem{remark}[thm]{Remark}

\newtheorem{prop}[thm]{Proposition}
\newtheorem{assumptions}[thm]{Assumptions}

\newcommand{\R}{\mathbb{R}}
 \newcommand{\E}{\mathbb{E}}
  \newcommand{\cX}{\mathbb{R}^d}
  \newcommand{\cY}{\mathbb{R}^\ell}
 \newcommand{\eps}{\epsilon}
\newcommand{\x}{x^{(\epsilon)}}
\newcommand{\y}{y^{(\epsilon)}}
\newcommand{\Lip}{\operatorname{Lip}}
\newcommand{\dist}{\operatorname{dist}}

\newcommand{\tildeJ}{{\tilde J}}

\numberwithin{equation}{section}

\begin{document}
\setlength{\baselineskip}{10pt}
\title{A Note on Diffusion Limits of\\
Chaotic Skew Product Flows}
\author{
I. Melbourne
\footnote{E-mail address: i.melbourne@surrey.ac.uk} \\
        Mathematics Department\\
        University of Surrey\\
	Guildford GU2 7XH, UK\\
and\\
        A.M. Stuart 
\footnote{E-mail address: a.m.stuart@warwick.ac.uk.} \\
        Mathematics Institute \\
        Warwick University \\
        Coventry CV4 7AL, UK
                    }

\date{September 30, 2010.  Updated 19 April 2015.\\ This paper
corrects an error in the published version of the paper
which appeared in {\bf Nonlinearity 24}(2011), 1361--1367.}

\maketitle

\begin{abstract}
We provide an explicit rigorous derivation of a diffusion
limit -- a stochastic differential equation with
additive noise -- from a deterministic skew-product flow. This
flow is assumed to exhibit time-scale separation and has
the form of a slowly evolving system driven by a 
fast chaotic flow.  Under mild assumptions on the fast flow,
we prove convergence to a stochastic differential equation as
the time-scale separation grows.   
In contrast to existing work, we do not require the flow to have good 
mixing properties.   As a consequence, our results incorporate a large class
of fast flows, including the classical Lorenz equations.

The updated version contains a correction to the proof of the main result, and removes an unnecessary large deviation assumption.
\end{abstract}

\section{Introduction}

There is considerable interest in understanding
how stochastic behaviour emerges from deterministic
systems, both in the mathematics and applications
literature. In this note we provide a simple
explicit construction of such emergent stochastic
behaviour in the setting of skew-product flows
exhibiting time-scale separation. 
We prove a diffusion limit for the following
ordinary differential equations (ODEs): 
\begin{subequations}
\label{eq:ode}
\begin{align}
\dot{x}^{(\eps)}& = \epsilon^{-1} f_0(\y)+f(\x,\y), \quad \x(0)=\xi,
\label{eq:odex}\\
\dot{y}^{(\eps)} & = \epsilon^{-2} g(\y), \quad \y(0)=\eta. 
\label{eq:odey}
\end{align}
\end{subequations}
Here $\x \in \cX, \y \in \cY$.
Roughly speaking we assume that the equation for $\y$ 
has a compact attractor $\Lambda\subset\cY$ supporting an invariant
measure $\mu$ and satisfying certain ``mild chaoticity'' assumptions.
These conditions are stated
precisely in Assumptions \ref{ass:1} below.   In addition, we assume
that $f_0$ should average to zero with respect to $\mu$.

Consider the stochastic differential equation (SDE) 
\begin{equation}
X(t)=\xi+\int_0^t F(X(s))\,ds+\sqrt \Sigma W(t).
\label{eq:isde}
\end{equation}
where $W$ is unit $d$-dimensional Brownian motion, $\Sigma$ is a $d\times d$
covariance matrix (depending on $f_0$ and $g$) and
$F(x)$ is the average of $f(x,\cdot)$ with respect to
the aforementioned invariant measure $\mu$. 
The goal of the note is to prove the following
limit theorem relating the solution $\x$ of \eqref{eq:ode} 
and $X$ of \eqref{eq:isde}. 
Throughout we use $\to_w$ to
denote weak convergence in the sense of probability
measures \cite{Bil68,billingsley}.

\begin{thm}
\label{thm-sdelimit}
Let Assumptions \ref{ass:1} hold and let
$\eta$ be a random variable distributed
according to the measure $\mu$ on the attractor $\Lambda\subset\cY$ and
fix any $\xi \in \cX$.
Then, almost surely with respect to $\eta$ and $W$,
there is a unique solution $(\x,\y) 
\in C^1([0,\infty); \cX \times \cY)$ of \eqref{eq:ode}
for each $\epsilon>0$,
and a unique solution $X \in C([0,\infty);\cX)$ of
\eqref{eq:isde}.
Furthermore
$\x\to_w X$ in $C([0,\infty),\cX)$ as $\epsilon\to0$.
\end{thm}

Throughout the note we make the following standing assumptions. 

\begin{assumptions}
\label{ass:1}
The differential equations \eqref{eq:ode} satisfy the
following:

\begin{enumerate}

\item Equation \eqref{eq:odey} with $\epsilon=1$ has a compact
invariant set $\Lambda$, $\eta\in\Lambda$, and there
is an invariant probability measure $\mu$ 
supported on $\Lambda$;
expectation with respect to this measure is denoted by $\E$.

\item   The vector fields $g:\Lambda\to\R^\ell$ and $f_0:\Lambda\to\R^d$
are locally Lipschitz, and 
the vector field $f: \cX \times \Lambda \to \R^d$ is bounded
and Lipschitz with uniform Lipschitz constant $L$.

\item The vector field $f_0: \Lambda \to \R^d$ averages to $0$ under $\mu: \E f_0=0.$

\item Define $W_n(t)=n^{-\frac12}\int_0^{nt}f_0(y^{(1)}(\tau))\,d\tau$, 
for $t\ge0$.  Fix any $T>0$.
We assume the weak invariance principle (WIP), namely that $W_n\to_w \sqrt \Sigma W$ in $C([0,T],\R^d)$ as $n\to\infty$
for unit $d$-dimensional Brownian motion $W$ and some covariance matrix 
$\Sigma,$ independent of $T$.

\end{enumerate}
\end{assumptions}

\begin{remark}
(a) The regularity conditions on $f,f_0,g$ in assumption 2 guarantee global existence and uniqueness of solutions to 
the ODEs \eqref{eq:ode} and the SDE \eqref{eq:isde} 
for all positive time and all initial conditions $\xi\in\cX$, $\eta\in\Lambda$.    We note that
uniformity of the Lipschitz constant for $f$ is 
automatic in $y$ since $\Lambda$ is compact.

\noindent(b)     
Assumption 4 holds for a large class of flows.
In particular, the WIP
is proved in \cite{MelNic} for flows that have a Poincar\'e map modelled by a 
Young tower~\cite{Young98,Young99} with summable tails.   This includes Anosov and  Axiom~A flows,
nonuniformly hyperbolic flows such as H\'enon-like flows (where the Poincar\'e map has a H\'enon-like attractor),
and Lorenz attractors~\cite{HM07} (including the case of the classical parameter values in~\cite{Lorenz63}).
In this  class of examples, the Poincar\'e map has good statistical properties 
and limit laws such as the WIP transfer to the flow~\cite{MT04}.
\end{remark}

There are two main routes leading to emergent
stochastic behaviour in deterministic systems.
The first is through the elimination of a large
number of degrees of freedom, and the reliance
on the central limit theorem to provide fluctuations and 
the second is through time-scale separation; see
\cite{GKS04} for an overview.
The first mechanism does not require any assumption
of chaotic behaviour and may even be observed
in large systems of linear oscillators; work
in this area was initiated in \cite{FKM65} and
more recent work includes \cite{KSTT02,ave09}.
The second mechanism relies on the presence of
some fast chaotic dynamics to induce white noise 
and has a long history in the applied literature; we
mention, in particular, the work in
\cite{Bec90,JKRH01,MT00,MTV06}. Our work
provides a very simple scenario in which the
second mechanism may be used, provably, to establish
emergent stochastic dynamics. We anticipate that
the basic ideas would apply to a far larger class
of problems as indicated, for example, by the
program outlined in
\cite{mac10}. Moreover the basic mechanism that
underlies the work in this note was identified
and studied in the seminal paper \cite{PK74}. However
the conditions in that paper can be hard to
verify for specific ordinary differential
equations. In contrast our construction holds for explicit systems on $\cY$ 
such as the classical Lorenz equations.

An important aspect of our theory is that we require no knowledge of mixing properties of the flow,
In contrast, previous rigorous results in the literature required strong assumptions on the mixing properties
of the flow.   See~\cite{Dolgopyat05} for the most powerful results in this direction where it is required
that the flow has stretched exponential decay of correlations.    Even for Anosov flows this has been proved
only in very special cases~\cite{Chernov98, Dolgopyat98a, Liverani04}.   Superpolynomial decay has been proved for 
typical Anosov and Axiom A flows~\cite{Dolgopyat98b, FMT07} and typical nonuniformly hyperbolic flows governed by Young 
towers~\cite{M07,M09} but only for very smooth observables; this smoothness would
have to be imposed on $f_0$.    For the Lorenz equations there are currently no results at all on rates of mixing
(though superpolynomial decay holds for typical nearby flows by~\cite{M09}).

\section{Diffusion Limit}

We now prove the diffusion limit
contained in Theorem \ref{thm-sdelimit}.
The method of proof generalizes that
described in Chapter 18 of \cite{PavlSt06b}
for homogenization in SDEs with additive noise and
a skew-product form.

\begin{prop}
\label{prop-WIP}
Let $(\x(t),\y(t))$ denote the solution to
\eqref{eq:ode} with $f \equiv 0$, $\xi=0$ and 
with $\eta$ a random variable distributed according to the measure 
$\mu$ on $\Lambda$.
Let $T>0$.  Then $\x\to_w \sqrt \Sigma W$ in $C([0,T],
\cX)$ as $\epsilon\to0$.
Here, $W$ is unit $d$-dimensional Brownian motion  and the covariance
matrix $\Sigma$ is independent of $T$.
\end{prop}

\begin{proof}
Note that $y^{(1)}(t)$ is the solution to the IVP $\dot y=g(y)$, $y(0)=\eta$.
Define $W_n(t)=n^{-\frac12}\int_0^{nt}f_0(y^{(1)}(\tau))\,d\tau$, for $t\in[0,T]$.
By the WIP, $W_n\to_w \sqrt \Sigma W$ in $C([0,T],\cX)$ as $n\to\infty$.

Now $\y(t)=y^{(1)}(t\epsilon^{-2}).$ Hence 
\[
\x(t) =\epsilon^{-1}\int_0^t f_0(\y(s))\,ds = 
\epsilon\int_0^{t\epsilon^{-2}} f_0(y^{(1)}(\tau))\,d\tau. 
\]
Writing $n=\epsilon^{-2}$, we obtain
$\x(t) =W_n(t)$ and the result follows.
\end{proof}
\vspace{0.1in}

\noindent {\em Proof of Theorem \ref{eq:ode}}
To prove weak convergence on $[0,\infty)$, it suffices to establish weak convergence on $[0,T]$
for each fixed  $T>0$.   

Write $W^{(\eps)}(t)=\int_0^t \frac{1}{\eps} f_0(\y(s))\,ds$.
By integrating the $\x$ equation we have
\begin{align*}
\x(t)&=\xi+\int_0^t \frac{1}{\eps} f_0(\y(s))\,ds+\int_0^t f(\x(s),\y(s))\,ds\\
&=\xi+W^{(\eps)}(t) +\int_0^t F(\x(s))\,ds +Z^{(\eps)}(t)
\end{align*}
where
$$Z^{(\eps)}(t)=\int_0^t \bigl( f(\x(s),\y(s))-F(\x(s)) \bigr) ds.$$
We show below that $Z^{(\eps)}\to 0$ 
in probability in $C([0,T],\cX)$.
(That is, for any $c>0$ there exists
$\eps_0>0$ such that $\mu(\max_{[0,T]}|Z^{(\eps)}|)>c)<c$
for all $\eps\in(0,\eps_0)$.
By Proposition \ref{prop-WIP}, 
$W^{(\eps)} \to_w \sqrt\Sigma W$ in $C([0,T],\cX)$. 
It follows that $W^{(\eps)}+Z^{(\eps)} \to_w \sqrt\Sigma W$ 
in $C([0,T],\cX).$
Now consider the continuous map
$\mathcal{G}:C([0,T],\cX)\to C([0,T],\cX)$ given by $\mathcal{G}(u)=v$
where $v$ is the unique solution to the integral equation
$$v(t)=\xi+u(t)+\int_0^t F(v(s))\,ds.$$  
Define $v^{(\eps)}= 
\mathcal{G}(W^{(\eps)}+Z^{(\eps)})$.
Since continuous maps preserve weak convergence, it follows that 
$v^{(\eps)}\to_w \mathcal{G}(\sqrt\Sigma W)=X$.   But $v^{(\eps)}=x^{(\eps)}$ by uniqueness of solutions,
so $x^{(\eps)}\to_w X$ as required.

It remains to show the convergence in probability of
$Z^{(\eps)}$ to $0$ in $C([0,T],\cX)$.
Define
$g(x,y)=f(x,y)-F(x)$ and note that $|g|_\infty\le 2|f|_\infty$
and $\Lip(g)\le 2L$.   Then
$Z^{(\eps)}(t) =\int_0^t g(\x(s),\y(s))\,ds$.
Let $N=[t/\eps^{3/2}]$ and write $Z^{(\eps)}(t)=Z^{(\eps)}(N\eps^{3/2})+I_0$
where $I_0=\int_{N\eps^{3/2}}^t g(\x(s),\y(s))\,ds$.  
We have 
\begin{align} \label{eq-I0}
|I_0|\le (t-N\eps^{3/2})|g|_\infty \le 2|f|_\infty\eps^{3/2}.
\end{align}

We now estimate $Z^{(\eps)}(N\eps^{3/2})$ as follows:
\begin{align*}
Z^{(\eps)}(N\eps^{3/2})&= \sum_{n=0}^{N-1} \int_{n\eps^{3/2}}^{(n+1)\eps^{3/2}} 
g(\x(s),\y(s))\,ds\\
& =\sum_{n=0}^{N-1}
\int_{n\eps^{3/2}}^{(n+1)\eps^{3/2}}\bigl(g(\x(s),\y(s))-g(\x(n\eps^{3/2}),\y(s))\bigr)\,ds \\ &\qquad\qquad +\sum_{n=0}^{N-1}
\int_{n\eps^{3/2}}^{(n+1)\eps^{3/2}} g(\x(n\eps^{3/2}),\y(s))\,ds \\
& = I_1+I_2.
\end{align*}

For $s \in [n\eps^{3/2},(n+1)\eps^{3/2}]$, we have
$|\x(s)-\x(n\eps^{3/2})| \le (|f_0|_{\infty}+|f|_{\infty})\eps^{1/2}$.
Hence 
\begin{align} \label{eq-I1}
|I_1|\le N\eps^{3/2} \Lip(g)(|f_0|_{\infty}+|f|_{\infty})\eps^{1/2}\le 
2L(|f_0|_{\infty}+|f|_{\infty})T\eps^{1/2}.
\end{align}

Next,
\begin{align*}
I_2 & =\sum_{n=0}^{N-1} \int_{n\eps^{3/2}}^{(n+1)\eps^{3/2}}g(\x(n\eps^{3/2}),\y(s))\,ds
=\eps^{3/2}\sum_{n=0}^{N-1}J_n,
\end{align*}
where
\begin{align*}
J_n  &= \eps^{-3/2}\int_{n\eps^{3/2}}^{(n+1)\eps^{3/2}} g(\x(n\eps^{3/2}),\y(s))\,ds
 \\ & = \eps^{1/2}\int_{n\eps^{-1/2}}^{(n+1)\eps^{-1/2}}
g(\x(n\eps^{3/2}),y^{(1)}(s))\,ds.
\end{align*}
Hence
\begin{align} \label{eq-J}
|I_2|\le \eps^{3/2}\sum_{n=0}^{[T\eps^{-3/2}]-1}|J_n|.
\end{align}

For $u\in\R^d$ fixed, we define
\[
\tildeJ_n(u)=\eps^{1/2} 
\int_{n\eps^{-1/2}}^{(n+1)\eps^{-1/2}}  g(u,y^{(1)}(s))\, ds
=\eps^{1/2}\int_{n\eps^{-1/2}}^{(n+1)\eps^{-1/2}}  
A_u\circ \phi_s\,ds,
\]
where $A_u(y)=g(u,y)$.
Note that $\tildeJ_n(u)=\tildeJ_0(u)\circ \phi_{n\eps^{-1/2}}$,
and so
$\E|\tildeJ_n(u)| =\E|\tildeJ_0(u)|$.
By the ergodic theorem,
$\E|\tildeJ_0(u)|\to0$ as $\eps\to0$ for each $u$.

Let $Q>0$ and write $I_2=K_{Q,1}+K_{Q,2}$ where
\begin{align*}
& K_{Q,1}=I_21_{B_\eps(Q)}, \quad
K_{Q,2}=I_21_{B_\eps(Q)^c}, \quad
 B_\eps(Q)=\bigl\{\max_{[0,T]}|\x|\le Q\bigr\}.
\end{align*}

For any $a>0$, there exists a finite subset $S\subset\R^d$ such that
$\dist(x,S)\le a/(2L)$ for any $x$ with $|x|\le Q$.
Then for all $n\ge0$, $\eps>0$,
\[
1_{B_\eps(Q)}|J_n|\le \sum_{u\in S}|\tildeJ_n(u)|+a.
\]
Hence by~\eqref{eq-J},
\begin{align*}
 \E\max_{[0,T]}|K_{Q,1}|
 & \le \eps^{3/2}\sum_{n=0}^{[T\eps^{-3/2}]-1}
\sum_{u\in S}\E|\tildeJ_n(u)|+Ta
\\ & = \eps^{3/2}\sum_{n=0}^{[T\eps^{-3/2}]-1}
\sum_{u\in S}\E|\tildeJ_0(u)|+Ta
  \\ & \le T \sum_{u\in S}\E|\tildeJ_0(u)|+Ta.
\end{align*}
Since $a>0$ is arbitrary, we obtain for each fixed $Q$ that
$\max_{[0,T]}|K_{Q,1}|\to0$ in $L^1$, and hence in probability, as $\eps\to0$.

Next, since $\x-W^{(\eps)}$ is bounded on $[0,T]$, for $Q$ sufficiently large
\[
\mu\bigl\{\max_{[0,T]}|K_{Q,2}|>0\bigr\}\le 
\mu\bigl\{\max_{[0,T]}|\x|\ge Q\bigr\}\le 
\mu\bigl\{\max_{[0,T]}|W^{(\eps)}|\ge Q/2\bigr\}.
\]
Fix $c>0$.  Increasing $Q$ if necessary, we can arrange that
$\mu\{\max_{[0,T]}|\sqrt\Sigma W|\ge Q/2\}<c/4$.
By the continuous mapping theorem,
$\max_{[0,T]}|W^{(\eps)}|\to_d \max_{[0,T]}|\sqrt\Sigma W|$.
Hence there exists $\eps_0>0$ such that
$\mu\{\max_{[0,T]}|W^{(\eps)}|\ge Q/2\}<c/2$ for all $\eps\in(0,\eps_0)$.
For such $\eps$,
\[
\mu\bigl\{\max_{[0,T]}|K_{Q,2}|>0\bigr\}<c/2.
\]
Shrinking $\eps_0$ if necessary, we also have that
$\mu\{\max_{[0,T]}|K_{Q,1}|>c/2\}<c/2$.
Hence $\mu\{\max_{[0,T]}|I_2|>c\}<c$, and so 
$\max_{[0,T]}|I_2|\to0$ in probability.  
Combining this with estimates~\eqref{eq-I0} and~\eqref{eq-I1}, we obtain that 
$\max_{[0,T]}|Z^{(\eps)}|\to0$ in probability as required.

\section{Conclusions}

The construction in this paper shows how some new
ideas in the theory of dynamical systems can be
used to prove a homogenization principle in ODEs, 
leading to emergent stochastic
behaviour. The arguments are very straightforward, and
are given only in the case of additive noise. However
in the situation where the limiting SDE is one dimensional 
the ideas of Sussmann \cite{suss78} 
can be used to derive a limiting SDE in which noise
appears multiplicatively. Generalizing these ideas to
skew product flows where the SDE is of higher
dimension will require the theory of rough paths \cite{Lyons98}
and is the subject of ongoing work.

Finally a comment on the differences
between \cite{PavlSt06b}
{\em homogenization} and {\em averaging}
in ODE systems like \eqref{eq:ode}.
There is current interest \cite{aet10}
in the derivation of averaging principles for systems
of ODEs exhibiting three time scales of order
${\cal O}(\eps^{-2}), {\cal O}(\eps^{-1})$ and
${\cal O}(1)$. The motivation is the construction
of efficient numerical schemes for computation
of the averaged solution, which is deterministic.
Theorem \ref{thm-sdelimit}, which also concerns
the limiting behaviour of a system containing
three time-scales, corresponds to a
homogenization principle with
a stochastic limit, 
rather than an averaging principle 
with deterministic limit. Thus our work
provides an example of a three scale
system  
for which an effective deterministic
averaged equation cannot exist.

\vspace{0.2in}

\noindent{\bf Acknowledgements.} The authors are grateful
to Niklas Br\"annstr\"om and Matthew Nicol for helpful discussions. 
AMS is grateful to EPSRC and ERC for financial support.
The research of IM was supported in part by  EPSRC Grant EP/F031807/01

\def\cprime{$'$} \def\cprime{$'$} \def\cprime{$'$} \def\cprime{$'$}
  \def\cprime{$'$} \def\cprime{$'$} \def\cprime{$'$}
  \def\Rom#1{\uppercase\expandafter{\romannumeral #1}}\def\u#1{{\accent"15
  #1}}\def\Rom#1{\uppercase\expandafter{\romannumeral #1}}\def\u#1{{\accent"15
  #1}}\def\cprime{$'$} \def\cprime{$'$} \def\cprime{$'$} \def\cprime{$'$}
  \def\cprime{$'$} \def\cprime{$'$}

\end{document}